\documentclass[11pt,epsf]{article}
\usepackage{graphicx}
\usepackage{amsthm}
\usepackage{amsfonts}
\usepackage{amssymb}
\title{Uniform bounds on locations of zeros of partial theta function}
\author{Vladimir Petrov Kostov\\ 
Universit\'e de Nice, 
Laboratoire de Math\'ematiques, Parc Valrose,\\ 06108 Nice Cedex 2, France,  
e-mail: kostov@math.unice.fr} 
\date{}
\newtheorem{tm}{Theorem}

\newtheorem{rems}[tm]{Remarks}
\newtheorem{lm}[tm]{Lemma}

\newtheorem{nota}[tm]{Notation}
\begin{document} 
\maketitle 
\begin{abstract}
We consider the partial theta 
function $\theta (q,z):=\sum _{j=0}^{\infty}q^{j(j+1)/2}z^j$, where  
$(q,z)\in \mathbb{C}^2$, $|q|<1$. We show that for any 
$0<\delta _0<\delta <1$, there exists $n_0\in \mathbb{N}$ such that 
for any $q$ with $\delta _0\leq |q|\leq \delta$ and for any 
$n\geq n_0$ the function $\theta$ has exactly $n$ 
zeros with modulus $<|q|^{-n-1/2}$ counted with multiplicity.\\   


\end{abstract}

\section{Introduction}

We consider the bivariate series $\theta (q,z):=\sum _{j=0}^{\infty}q^{j(j+1)/2}z^j$, 
where $(q,z)\in \mathbb{C}^2$, $|q|<1$. This series defines a {\em partial 
theta function}. The terminology is explained by the fact that the Jacobi 
theta function is defined by the series $\sum _{j=-\infty}^{\infty}q^{j^2}z^j$ and 
the following equality holds true: 
$\theta (q^2,z/q)=\sum _{j=0}^{\infty}q^{j^2}z^j$. The word ``partial'' 
is justified by the summation in $\theta$ ranging from $0$ to $\infty$ and 
not from $-\infty$ to $\infty$. In what follows we consider $z$ as a 
variable and $q$ as a parameter. For each fixed value of the parameter $q$ 
the function $\theta$ is an entire function in the variable $z$.

The function $\theta$ finds applications in various domains, such as 
statistical physics and combinatorics (see \cite{So}), 
Ramanujan type $q$-series (see \cite{Wa}), the theory 
of (mock) modular forms (see \cite{BrFoRh}), asymptotic analysis 
(see \cite{BeKi}), and also in problems concerning 
real polynomials in one variable 
with all roots real (such polynomials are called {\em hyperbolic}, 
see \cite{Ha}, \cite{Hu}, \cite{Pe}, \cite{Ost}, 
\cite{KaLoVi}, \cite{KoSh} and \cite{Ko2}). Other facts about $\theta$ can be 
found in~\cite{AnBe}.

The zeros of $\theta$ depend on the parameter $q$. For some values of $q$ 
(called {\em spectral}) confluence of zeros occurs, so it would 
be correct to regard the zeros as multivalued functions of $q$;  
about the spectrum of 
$\theta$ see \cite{KoSh}, \cite{Ko8} and \cite{Ko9}. 

\begin{nota}
{\rm We denote by $\mathbb{D}_{\rho }$ the open disk in the $q$-space 
centered at $0$ and of radius $\rho$, by $\mathcal{C}_{\rho }$ the corresponding 
circumference, and by $A_{\delta _0,\delta }$ the closed annulus 
$\{ q\in \mathbb{C}\, |\, \delta _0\leq |q|\leq \delta \}$.}
\end{nota}
In the present paper we prove the following theorem:

\begin{tm}\label{maintm}
For any couple of numbers $(\delta _0,\delta )$ such that 
$0<\delta _0<\delta <1$, there exists $n_0\in \mathbb{N}$ such that 
for any $q\in A_{\delta _0,\delta }$ and for any 
$n\geq n_0$ the function $\theta$ has exactly $n$ 
zeros in $\mathbb{D}_{|q|^{-n-1/2}}$ counted with multiplicity.
\end{tm}

\begin{rems}\label{remarks}
{\rm (1) The proof of the theorem 
is based on a comparison between $\theta$ and the function 
\begin{equation}\label{functionu}
u(q,z):=\prod _{\nu =1}^{\infty}(1+q^{\nu}z)
\end{equation}
We use the equality 
\begin{equation}\label{equalityu}
u=\sum _{j=0}^{\infty}q^{j(j+1)/2}z^j/(q;q)_j~,
\end{equation} 
where $(q;q)_j:=(1-q)(1-q^2)\cdots (1-q^j)$ is the 
$q$-Pochhammer symbol; it follows directly 
from Problem I-50 of \cite{PoSz} (see pages 9 and 186 of \cite{PoSz}). 
The analog of the above theorem for the 
{\em deformed exponential function} $\sum _{j=0}^{\infty}q^{j(j+1)/2}z^j/j!$ 
is proved in a non-published text by A.~E.~Eremenko using a different method.

(2) For $q$ close to $0$ the zeros of $\theta$ 
are of the form $-q^{-\ell}(1+o(1))$, $\ell \in \mathbb{N}$, 
see more details about 
this in \cite{Ko4}, \cite{Ko6} and \cite{Ko7}.}
\end{rems}

{\bf Acknowledgement.} The author has discussed (electronically and directly) 
questions concerning the partial theta and the deformed exponential function 
with A.~Sokal, A.~E.~Eremenko, B.~Z.~Shapiro, I.~Krasikov 
and J.~Forsg{\aa}rd to 
all of whom he expresses his most sincere gratitude.

\section{Proofs}

\begin{proof}[Proof of Theorem~\ref{maintm}]
It is shown in \cite{Ko4} that for $0<|q|\leq 0.108$ the zeros of $\theta$ 
can be expanded in convergent Laurent series. 
Recall that the function $u$ (defined by (\ref{functionu})) satisfies equality 
(\ref{equalityu}), i.e. the zeros of $u$ 
are the numbers $-q^{-\ell}$, $\ell \in \mathbb{N}$. 
We show that 
for $n\in \mathbb{N}$ sufficiently large the functions $u$ and $\theta$ 
have one and the same number of zeros in the open 
disk $\mathbb{D}_{|q|^{-n-1/2}}$. 
To this end we show that for the 
restrictions $u^0$ and $\theta ^0$ 
of $u$ and $\theta$ to the circumference $\mathcal{C}_{|q|^{-n-1/2}}$ 
one has $|u^0-\theta ^0/(q;q)_n|<|u^0|$ after which 
we apply the Rouch\'e theorem. 

For $0<|q|\leq 0.108$ 
one can establish a bijection between the zeros of $\theta$ and $u$, because 
their 
$\ell$th zeros are of the form $-q^{-\ell}(1+o(1))$ and the moduli of the zeros 
increase with $\ell$, see part (2) of Remarks~\ref{remarks}. 

Set $P_k(|q|):=\prod_{\ell =0}^{k}(1-|q|^{\ell +1/2})$, $k\in \mathbb{N}\cup \infty$. 
For $|u^0|$ one obtains the estimation 

\begin{equation}\label{equ0}
|u^0|\geq |q|^{-n^2/2}P_{n-1}(|q|)P_{\infty}(|q|)>
|q|^{-n^2/2}(P_{\infty}(|q|))^2\geq |q|^{-n^2/2}(P_{\infty}(\delta ))^2~.
\end{equation} 
Indeed, for $|z|=|q|^{-n-1/2}$ one can set $z:=|q|^{-n-1/2}\omega$, $|\omega |=1$. 
For $1\leq \nu \leq n$ (resp. for $\nu >n$), 
the factor $(1+q^{\nu}z)$ in (\ref{functionu}) 
is of the form $(1-|q|^{-\ell -1/2}\omega _{\ell})$, 
where $\ell =n-\nu$ and $|\omega _{\ell}|=1$ (resp. of the form 
$(1-|q|^{\ell +1/2}\omega ^*_{\ell})$, where $\ell =\nu -n-1$ and 
$|\omega ^*_{\ell}|=1$). Thus 

$$
u(q,|q|^{-n-1/2}\omega ^{-n-1/2})=
\prod_{\ell =0}^{n-1}(1-|q|^{-\ell -1/2}\omega _{\ell})
\prod _{\ell =0}^{\infty}(1-|q|^{\ell +1/2}\omega ^*_{\ell})~.$$
The first of the factors in the right-hand side can be represented in the form 
$|q|^{-n^2/2}\tilde{\omega}\prod_{\ell =0}^{n-1}(1-|q|^{\ell +1/2}\omega ^{**}_{\ell})$ 
with $|\tilde{\omega}|=|\omega ^{**}_{\ell}|=1$. Therefore

$$u(q,|q|^{-n-1/2}\omega ^{-n-1/2})=
|q|^{-n^2/2}\tilde{\omega}\prod_{\ell =0}^{n-1}(1-|q|^{\ell +1/2}\omega ^{**}_{\ell})
\prod _{\ell =0}^{\infty}(1-|q|^{\ell +1/2}\omega ^*_{\ell})~.
$$  
The modulus of the 
right-hand side is minimal for $\omega ^*_{\ell}=\omega ^{**}_{\ell}=1$ 
in which case one obtains the leftmost inequality in (\ref{equ0}).

Consider the monomial $\beta _j:=\alpha _jz^j$ in the 
series $u-\theta /(q;q)_n$. 
Hence for $j=n$ it vanishes and for $j>n$ one has 

$$\begin{array}{ccl}\alpha _j&=&q^{j(j+1)/2}(1/(q;q)_j-1/(q;q)_n)=
q^{j(j+1)/2}U_{j,n}~~,~~{\rm where}\\ \\ 
U_{j,n}&:=&(1-\prod_{\ell =n+1}^{j}(1-q^{\ell}))/(q;q)_j~,\end{array}$$
so for $|z|=|q|^{-n-1/2}$ 
one has $|\beta _j|=|q|^{-n^2/2+(j-n)^2/2}|U_{j,n}|$. 
One can observe that $U_{j,n}=q^{n+1}+O(q^{n+2})$. 
Set
$$U_{j,n}:=\sum _{\nu \geq n+1}u_{j,n;\nu }q^{\nu}~~{\rm and}~~ 
U:=((\prod_{\ell =1}^{\infty}(1+q^{\ell}))-1)/(q;q)_{\infty}=
\sum _{\nu =1}^{\infty}u_{\nu}q^{\nu}~.$$ 
The Taylor series of $U$ converges for $|q|<1$ because the infinite products 
defining $U$ converge. Clearly 
$u_{j,n;\nu}\in \mathbb{Z}$, $u_{\nu}\in \mathbb{N}$ (because all coefficients of 
the series $1/(q;q)_j$ and $1/(q;q)_{\infty}$ are positive integers) 
and $u_{j,n;n+1}=u_1=1$.

The following lemma explains in what sense the series $U$ 
majorizes the series $U_{j,n}$.

\begin{lm}\label{lmU}
One has $|u_{j,n;n+\nu }|\leq u_{\nu}$, $\nu \in \mathbb{N}$. 
\end{lm}

Before proving Lemma~\ref{lmU} (the proof is given at the end of the paper) 
we continue the proof of Theorem~\ref{maintm}.


Set $R(|q|):=\sum _{j>n}|q|^{(j-n)^2/2}$. The following inequality results 
immediately from the lemma:

\begin{equation}\label{eqlmU}
Z_1:=\sum _{j>n}|\beta _j|\leq |q|^{-n^2/2}|q|^nU(|q|)
R(|q|)\leq |q|^{-n^2/2}\delta ^nU(\delta )R(\delta )~~.
\end{equation}
The first condition which we 
impose on the choice of $n$ is the 
following inequality to be fulfilled:

\begin{equation}\label{condition1}
\delta ^nU(\delta )R(\delta )<(P_{\infty}(\delta ))^2/4~.
\end{equation}

For $j<n$ and $|z|=|q|^{-n-1/2}$ one has $|\beta _j|=|q|^{-n^2/2+(j-n)^2/2}|\tilde{U}_{j,n}|$, where  

\begin{equation}\label{formula}
\tilde{U}_{j,n}:=(\prod_{\ell =j+1}^n(1-q^{\ell})-1)/(q;q)_n~.
\end{equation} 
Hence 
$|\tilde{U}_{j,n}|\leq T(|q|):=(\prod_{\ell =1}^{\infty}(1+|q|^{\ell})+1)/(|q|;|q|)_{\infty}$ and 

\begin{equation}\label{eqbeta}
|\beta _j|\leq |q|^{-n^2/2}|q|^{(j-n)^2/2}T(\delta )
\end{equation}
Choose $m\in \mathbb{N}$ such that 
$T(\delta )\sum _{s=m}^{\infty}\delta ^{s^2/2}\leq (P_{\infty}(\delta ))^2/4$. 
Inequality (\ref{eqbeta}) implies that 

\begin{equation}\label{eqbeta1}
Z_2:=\sum _{j=0}^{n-m}|\beta _j|\leq |q|^{-n^2/2}(P_{\infty}(\delta ))^2/4
\end{equation}
Notice that for $n<m$ the above sum is empty 
and the inequality trivially holds true. 

The finite sum 

\begin{equation}\label{Z3}
Z_3:=\sum _{j=n-m+1}^{n-1}|\beta _j|
\end{equation} is of the form 
$|q|^{-n^2/2}O(|q|^n)$. Indeed, consider formula (\ref{formula}).
There exists $M>0$ depending only on $\delta _0$ and $\delta$ such that 
$$0<|1/(q;q)_n|\leq 1/(|q|;|q|)_n<1/(|q|;|q|)_{\infty}\leq M~~{\rm for}~~\delta _0\leq |q|\leq \delta ~.$$ 
Thus 
$$|\tilde{U}_{j,n}|\leq M(\prod _{\ell =j+1}^n(1+|q|^{\ell})-1)~.$$
The index $j$ can take only the values $n-m+1$, $\ldots$, $n-1$. 
In the last product each monomial 
$|q|^{\ell}$ can be represented in the form $|q|^n|q|^{\ell -n}$, 
where $\ell -n=2-m$, $\ldots$, $0$. 
The modulus of each factor $|q|^{\ell -n}$ 
is not larger than $1/\delta _0^{\max (0,m-2)}$. 
Therefore 
$$|\tilde{U}_{j,n}|\leq M((1+|q|^n/\delta _0^{\max (0,m-2)})^{m-1}-1)=O(|q|^n)~.$$
The sum $Z_3$ (see (\ref{Z3})) can be 
made less than 
$|q|^{-n^2/2}(P_{\infty}(\delta ))^2/4$ by choosing $n$ large enough. Thus 
inequalities (\ref{equ0}), (\ref{eqlmU}) and (\ref{eqbeta1}) yield 

$$|u^0-\theta ^0/(q;q)_n|\leq Z_1+Z_2+Z_3
\leq (3/4)|q|^{-n^2/2}(P_{\infty}(\delta ))^2<
|q|^{-n^2/2}(P_{\infty}(\delta ))^2
\leq |u^0|$$
which proves the theorem.
\end{proof}

\begin{proof}[Proof of Lemma~\ref{lmU}]
We first compare the coefficients of the series 

$$\prod_{\ell =p}^{r}(1+q^{\ell})-1=\sum _{\nu \geq p}\gamma^1 _{\nu}q^{\nu}~~ 
{\rm and}~~\prod_{\ell =p}^{r}(1-q^{\ell})-1=
\sum _{\nu \geq p}\gamma^2 _{\nu}q^{\nu}~~,~~p\leq r~.$$ 
They are obtained respectively as a sum of the non-negative coefficients 
of monomials and as a linear 
combination of the same coefficients some 
of which are taken with the $+$ and the rest with the $-$ sign. 
Therefore $\gamma ^1_{\nu}\geq |\gamma^2_{\nu}|$, 
$\nu \geq p$. This means that $|u_{j,n;\nu}|\leq v_{j,n;\nu}\leq v_{\infty ,n;\nu}$, 
where

$$V_{j,n}:=(\prod_{\ell =n+1}^{j}(1+q^{\ell})-1)/(q;q)_j=
\sum _{\nu \geq n+1}v_{j,n;\nu }q^{\nu}~~,~~V_{\infty ,0}=U~~{\rm and}~~v_{\infty ,0;\nu}=
u_{\nu}~.$$
To prove the lemma it suffices to show that 

\begin{equation}\label{star}
v_{\infty ,n;n+\nu}\leq v_{\infty ,0;\nu}~.
\end{equation} 
Consider 
the series $S_{r}:=\prod_{\ell =r+1}^{\infty}(1+q^{\ell})-1=
\sum _{\nu \geq r+1}s_{r;\nu}q^{\nu}$ for $r=0$ and $r=n$. 
Compare the coefficients $s_{0;\nu}$ and $s_{n;n+\nu}$. 
The coefficient $s_{0;\nu}$ is 
equal to the number of ways in which $\nu$ can be represented as a sum of 
distinct natural numbers forming an increasing sequence 
whereas $s_{n;n+\nu}$ is the number of ways in which $n+\nu$ can be 
represented as a sum of distinct natural numbers $\geq n+1$ forming an 
increasing sequence. Clearly $s_{n;n+\nu}\leq s_{0;\nu}$. 
This implies inequality (\ref{star}) and the lemma, because 
one has $V_{\infty ,r}=S_r/(q;q)_{\infty}$ 
and the coefficients of the series $1/(q;q)_{\infty}$ are all positive.
\end{proof}

\end{document}